\documentclass[draft,12pt,reqno,a4paper]{amsart}

\usepackage[margin=3cm,footskip=1cm]{geometry}

\usepackage{amssymb} 
\usepackage{amsmath} 
\usepackage{amsfonts}
\usepackage{amsthm} 
\usepackage{mathtools}
\usepackage{color}

\usepackage{enumerate} 
\usepackage[latin1]{inputenc}
\usepackage[shortlabels]{enumitem}
\usepackage{color}
\usepackage{graphicx} 
\usepackage{verbatim}

\newcommand{\supp}{\operatorname{supp}}

\newcommand{\dist}{\operatorname{dist}}

\newcommand{\N}{{\mathbb{N}}}

\newcommand{\R}{{\mathbb{R}}}

\newcommand{\C}{{\mathbb{C}}}

\DeclareFontFamily{U}{mathx}{\hyphenchar\font45}
\DeclareFontShape{U}{mathx}{m}{n}{
      <5> <6> <7> <8> <9> <10>
      <10.95> <12> <14.4> <17.28> <20.74> <24.88>
      mathx10
      }{}
\DeclareSymbolFont{mathx}{U}{mathx}{m}{n}
\DeclareFontSubstitution{U}{mathx}{m}{n}
\DeclareMathAccent{\widecheck}{0}{mathx}{"71}

\renewcommand\i{\mathrm{i}}

\newcommand{\p}{{\mathrm p}}

\renewcommand{\c}{{\mathrm c}}

\newcommand{\e}{{\mathrm e}}
\newcommand{\ess}{{\mathrm {ess}}}

\renewcommand{\d}{{\mathrm d}}

\newcommand{\pupo}{{\mathrm {pp}}}

\renewcommand{\Re}{\operatorname{Re}}

\DeclarePairedDelimiter\inp\langle\rangle

\newcommand\parb[2][]{#1 \big ( #2#1\big )}

\newcommand\parbb[2][]{#1 \Big ( #2#1\Big )}
 
 \newcommand{\pp}{{\mathrm {pp}}}
  
\newcommand{\mand}{\text{ \,and\, }}

\newcommand{\cas}{{\textrm {the Cauchy--Schwarz inequality }}}

\DeclarePairedDelimiter\ket{\lvert}{\rangle}
\DeclarePairedDelimiter\bra{\langle}{\rvert}

\DeclareMathOperator*{\slim}{s-lim}

\DeclarePairedDelimiter\abs\lvert\rvert
\DeclarePairedDelimiter\norm\lVert\rVert
\DeclarePairedDelimiter\set{\{}{\}}

\newcommand{\brf}{{\breve f}}

\newcommand{\brp}{{\breve \psi}}

\newcommand{\bD}{{\mathbf D}}

\newcommand{\bZ}{{\mathbf Z}}

\newcommand{\bX}{{\mathbf X}}

\newcommand{\vA}{{\mathcal A}}

\newcommand{\vG}{{\mathcal G}}

\newcommand{\vH}{{\mathcal H}}

\newcommand{\vL}{{\mathcal L}}

\newcommand{\vO}{{\mathcal O}}

\newcommand{\vT}{{\mathcal T}}

\theoremstyle{plain}
\newtheorem{thm}{Theorem}[section]

\newtheorem{lemma}[thm]{Lemma} \newtheorem{corollary}[thm]{Corollary}
\newtheorem{cond}[thm]{Condition}
\theoremstyle{definition}

\newtheorem*{remarks*}{Remarks}
\newtheorem*{remark*}{Remark}

\numberwithin{equation}{section}

\title {Asymptotic  completeness for short-range  $N$-body
    systems revisited}

\author{E. Skibsted} \address[E. Skibsted]{Institut for Matematik\\
Aarhus Universitet\\ Ny Munkegade 8000 Aarhus C, Denmark}
\email{skibsted@math.au.dk}

\begin{document}

\begin{abstract} 
We review Yafaev's approach to asymptotic  completeness for 
    systems of particles mutually interacting with  short-range
    potentials \cite{Ya1}. The theory is  based  on computation
    of commutators  with time-independent (mostly  bounded)  observables yielding a
    sufficient supply of Kato smoothness bounds. 
  \end{abstract}

\allowdisplaybreaks

\maketitle


\tableofcontents

\section{Introduction}\label{sec:Introduction}

We consider $n$-dimensional  particles $ j=1, \dots,  N$ interacting by short-range
pair-potentials
\begin{equation*}
	V_{ij} (x_i -x_j)=\vO(\abs{x_i -x_j}^{-\delta}), \quad \delta>1.
\end{equation*}
 The total 
Hamiltonian  reads 
\begin{equation*}
H=H^{a_{\max}}=-\sum_{j = 1}^N \frac{1}{2m_j}\Delta_{x_j} + \sum_{1 \le i<j \le N} V_{ij} (x_i -x_j),
\end{equation*} acting 
on $\vH=L^2(\bX)$, where
\begin{equation*}
 \bX=\set[\Big]{ x=(x_1,\dots ,x_N)\in \R^n\times \dots \times  \R^n\mid\sum_{j=1}^{N} m_j x_j =
0}. 
\end{equation*}
 For the cluster decompositions  $a\neq a_{\max}=\set{1, \dots,  N}$
there are  sub-Hamiltonians $H^a$  defined similarly
acting on   $\mathcal H^a=L^2(\bX^a)$ (more precisely given for each $a$ as a tensor sum of Hamiltonians for
the individual clusters of $a$).

Recall that 
for any  \emph{channel}
$\alpha =(a,\lambda^\alpha, u^\alpha)$, i.e.   $a$ be  a cluster
decomposition, $a\neq a_{\max}$, $\lambda^\alpha\in \R$ and  
$(H^a-\lambda^\alpha)u^\alpha=0$ for a normalized $u^\alpha\in
\mathcal H^a$,  the corresponding \emph{channel wave operators}
\begin{subequations}
 \begin{equation}\label{eq:wave_opq}
  W_\alpha^{\pm}=\slim_{t\to \pm\infty}\e^{\i tH}J_\alpha\e^{-\i
   t( -\Delta_{x_a}+\lambda^\alpha )},\quad J_\alpha \varphi
=u^\alpha\otimes
\varphi.
 \end{equation} The existence of the channel wave operators as well as
 their  completeness, i.e. the property  that their ranges span the
 absolutely continuous subspace  of $H$,
 \begin{equation}\label{eq:ACo}
   \Sigma_\alpha^\oplus\,\, R( W_\alpha^{\pm})=\vH_{\rm ac}(H),
 \end{equation} 
\end{subequations}
 is a famous problem in mathematical physics and since decades ago  definitively solved,  we
 refer to 
 \cite{SS, Gra, Ya1, Ta,  DG, Gri, HS, HS2, Is}  although this  list is not complete
 (and for the long-range case to  \cite{De, Zi, DG}). 

In this paper we 
 review Yafaev's approach to asymptotic  completeness \cite{Ya1}
using his vector field constructions (which was inspired by \cite{Gra}) and a 
Mourre estimate (based on  the Graf vector field). A notable virtue of
this approach is that it relies  on commutators  of
\emph{time-independent} observables only, which has applications for
stationary scattering theory \cite{Ya2}, \cite{Ya3} and  \cite{Sk2} (and
for long-range potentials as well, see  \cite{Sk1}). Several of the
other known proofs of asymptotic  completeness  are using time-dependent observables excluding
applications in  stationary scattering theory. 
 In this paper we will
not study the more refined  stationary scattering theory, however
one of the goals is to pursue an approach as close as possible to
the recently improved    stationary scattering theory  \cite{Sk2}. This is  partly
motivated by  aesthetics, attempting in this way to `align' the stationary and 
time-dependent methods. Another  goal is potentially to lay the ground for
 encompassing 
completeness  for $N$-body Schr\"odinger operators with certain
unbounded obstacles,   going beyond the bounded case treated by  Griesemer  \cite{Gri},  as
well as for 
some 
progress on the scattering   problem for $N$-body Schr\"odinger
operators with time-periodic short-range pair-potentials. Rather
unsatisfying very little
is known about the  latter
problem.

Our use of the Mourre
estimate deviates somewhat from the implementation in \cite{Ya1} in
that we use a new 
 `phase space partition of unity'. This amounts  to `preparing' an
initial localization of a given scattering state  in the same spirit as
appearing in 
many  papers, see 
for example 
\cite{SS}. While this may appear as an unnecessary complication compared
to   Yafaev's very efficient and elegant arguments, our  approach has
as indicated above a  vital
analogue 
in \cite{Sk2}. 

More precisely let  $\psi(t)= \e^{-\i tH}f(H)\psi$ denote an 
energy-localized scattering state, say  as defined for any   $f\in C^\infty_\c(\R)$
supported near a fixed real $\lambda$,  not be a threshold nor an
eigenvalue. Then we  use the Mourre estimate to find a suitable bounded
observable $P_+$, whose symbol is localized  to the `forward direction'
of the  phase space, such that  $\norm{\psi(t)- P_+\psi(t)}\to 0$ for $t\to
+\infty$. Secondly  we show that  Yafaev's observable
$M_{a_{\max}}=\Sigma_{a\neq a_{\max}}\,M_a$ effectively can be substituted, 
yielding for a suitable  bounded function $g_+\in C^\infty(\R)$, that
\begin{subequations}
\begin{equation}\label{eq:locIni}
  \norm[\Big]{\psi(t)- \sum_{a\neq
    a_{\max}}\,g_+(M_{a_{\max}})M_a \psi(t)}\to 0\text { for }t\to
+\infty.
\end{equation}
  
Once the initial localization \eqref{eq:locIni} is established asymptotic  completeness
follows by a standard   induction argument from the existence of the limits
\begin{equation}\label{eq:parta}
  \psi_a=\lim_{t\to +\infty}\,\e^{\i tH_a}g_+(M_{a_{\max}})M_a \psi(t), \quad a \neq a_{\max},
\end{equation}  
\end{subequations}
which in turn   is  shown by  commutator arguments essentially mimicking \cite{Ya1}, although with
additional minor  calculus issues.

\section{Preliminaries and organization of paper}\label{sec:preliminaries}

We explain an  abstract setting (generalizing the one of Section
\ref{sec:Introduction}), give an account of some basic
$N$-body  scattering  theory in this setting and we outline  the structure of the
paper. 
\subsection{$N$-body Hamiltonians, assumptions  and
  notation}\label{subsec:body Hamiltonians, limiting absorption
  principle  and notation}

First we explain our model and then we introduce a number of
useful general notation.

\subsubsection{Generalized $N$-body short-range  Hamiltonians}\label{subsubsec::Short-range N-body Hamiltonians}

Let $\bX$ be                    
a (nonzero) finite dimensional real inner product space,
equipped with a
finite family $\{\bX_a\}_{a\in \vA}$ of subspaces closed under intersection:
For any $a,b\in\mathcal A$ there exists $c\in\mathcal A$, denoted
$c=a\vee b$,  such that $\bX_a\cap\bX_b=\bX_c$.
 We
 order $\vA$ by  writing $a\leq b$ (or equivalently as $b\geq a$) if
$\bX_a\supseteq \bX_b$. 
It is assumed that there exist
$a_{\min},a_{\max}\in \vA$ such that 
$\bX_{a_{\min}}=\bX$ and 
$\bX_{a_{\max}}=\{0\}$. For convenience abbreviate  $a_0:=a_{\max}$. The subspaces $\bX_a$, $a\neq a_{\min} $,  are called  \emph{collision
 planes}. We will  use the notation $\vA'=\vA\setminus
  \{a_{0}\} $ and $d_a=\dim
\mathbf X_a$ for $a\in \vA'$. Let  $d=\dim
\mathbf X$. 
 
The $2$-body model (or more correctly named `the one-body model')  is
based on the structure 
 $\vA=\set{a_{\min},a_0}$. 
The scattering  theory for such models is well-understood,
in fact (here including long-range potentials) there are several
doable approaches, see for example \cite{DG,Is, Sk1} for  accounts  on time-dependent  and stationary
   scattering  theories. In this paper we  consider short-range potentials only, see the below Condition
   \ref{cond:smooth2wea3n12},  for which the scattering theory for the
   $2$-body as well as for the $N$-body model has a  canonical meaning
   (the long-range case is in this sense different).

Let $\bX^a\subseteq\bX$ be the orthogonal complement of $\bX_a\subseteq \bX$,
and denote the associated orthogonal decomposition of $x\in\bX$ by 
$$x=x^a\oplus x_a=\pi^ax\oplus \pi_ax\in\bX^a\oplus \bX_a.$$ 
The vectors $x^a$ and $x_a$ may be  called the \emph{internal
  component}  and the 
\emph{inter-cluster component}  of $x$, respectively. The momentum operator
$p=-\i \nabla$ decomposes similarly, $p=p^a\oplus p_a$. For $a\neq a_{\min} $ the  unit sphere  in $\mathbf X^{a}$ is denoted
by $\mathbf{S}^{a}$. 

A real-valued measurable function $V\colon\bX\to\mathbb R$ is 
a \textit{potential of many-body type} 
if there exist real-valued measurable functions
$V_a\colon\bX^a\to\mathbb R$ such that 
\begin{equation*}
V(x)=\sum_{a\in\mathcal A}V_a(x^a)\ \ \text{for }x\in\mathbf X.
\end{equation*} We take $V_{a_{\min}}=0$ and  impose throughout the paper the
following condition for    $a\neq a_{\min}$.  
\begin{cond}\label{cond:smooth2wea3n12}
    There exists $\mu\in (0,1/2)$  such that for all $a\in \vA\setminus\set{a_{\min}}$ the
    potential $V_a=V_a(x^a)$ fulfills:
    \begin{enumerate}
    \item \label{item:shortr}$V_a(-\Delta_{x^a}+1)^{-1}$ is compact. 
    \item \label{item:shortl}$\abs{x^a}^{1+2\mu}
      V_a(-\Delta_{x^a}+1)^{-1}$ is bounded.
\end{enumerate}
\end{cond}

 For any $a\in\vA$  we  introduce   associated  {Hamiltonians} $H^a$
 and $H_a$   as follows. 
For $a=a_{\min}$  we define
$H^{a_{\min}}=0$ on $\mathcal H^{a_{\min}}=L^2(\set{0})=\mathbb C $
and $H_{a_{\min}} =
p^2$ on $L^2(\bX)$, respectively. 
For $a\neq a_{\min}$ 
we let 
\begin{equation*}
V^a(x^a)=\sum_{b\leq a} V_{b}(x^b)
,\quad
x^a\in \bX^a,
\end{equation*} 
and  define  then 
\begin{equation*}
 H^a=-
\Delta_{x^a} +V^a\ \ 
\text{on }\mathcal H^a=L^2(\bX^a)\mand H_a=H^a \otimes I +I\otimes
p^2_a\text{ on }L^2(\bX).
\end{equation*} 
We  abbreviate 
\begin{align*}
V^{a_{0}}=V,\quad
 H^{a_{0}}=H,\quad 
 \mathcal H^{a_{0}}=\mathcal H=L^2(\bX).
\end{align*}
 The operator $H$ (with domain $
\mathcal D(H)=H^2(\mathbf X)$) is  the full
Hamiltonian of the $N$-body model, and its resolvent is denoted by
$R(z)=(H-z)^{-1}$; $z\in \C\setminus\R$.  The \textit{thresholds} of $H$ are by
definition the
eigenvalues of the  sub-Hamiltonians $H^a$; 
$a\in  \vA'$.
 Equivalently stated  the set of thresholds is 
\begin{equation*}
 \vT (H):= \bigcup_{a\in\vA'} \sigma_{\pupo}( H^a).
\end{equation*}
 This set 
is closed and countable. Moreover the set of non-threshold eigenvalues
  is discrete in $\R\setminus \vT (H)$,  and it 
 can only  accumulate  at 
$\vT (H)$.  
The essential spectrum is given by the formula
$\sigma_{\ess}(H)= \bigl[\min \vT(H),\infty\bigr)$. 
We introduce the notation $\vT_{\p}(H)=\sigma_{\pp}(H)\cup
  \vT(H)$, and more generally   $\vT_{\p}(H^a)=\sigma_{\pp}(H^a)\cup
  \vT(H^a)$ for  $a\in  \vA$ (with $\vT(H^{a_{\min}}):=\emptyset$).

\subsubsection{General 
  notation}\label{subsubsec:General 
  notation}

Any  function $f\in C^\infty_\c(\R)$ taking values in $[0,1]$ is
referred to as a  
\emph{support function}. Given  such
function  $f_1$ and  $f\in C^\infty_\c(\R)$  we write
$f_1\succ f$, if $f_1=1$ in a neighbourhood of $\supp f$.

If $T$ is a self-adjoint  operator on a
Hilbert space and $f$ is a support function we can represent the
operator $f(T)$ by the well known Helffer--Sj\"ostrand formula
\begin{align}\label{82a0}
  \begin{split}
 f(T) 
&=
\int _{\C}(T -z)^{-1}\,\mathrm d\mu_f(z)\, \text{ with}\\
&\mathrm d\mu_f(z)=\pi^{-1}(\bar\partial\tilde f)(z)\,\mathrm du\mathrm dv;\quad 
z=u+\i v.   
  \end{split}
\end{align} Here $\tilde{f}$ is an `almost analytic' extension of
$f$, which in this case may be taken compactly supported. The
formula \eqref{82a0} extends to more general  classes of  functions
and  serves as a standard tool for commuting  operators. It will be
important for this issue in the present  paper as well,
however some details will be left out of our exposition. Rather we
prefer 
at
several places to 
refer  to explicit results from \cite[Subsections 6.2--6.4]{Sk1}, all
of them 
be convenient   calculus applications of \eqref{82a0}.

Consider   and fix $\chi\in C^\infty(\mathbb{R})$ such that 
\begin{align*}
\chi(t)
=\left\{\begin{array}{ll}
0 &\mbox{ for } t \le 4/3, \\
1 &\mbox{ for } t \ge 5/3,
\end{array}
\right.
\quad
\chi'\geq  0,
\end{align*} and such that the following properties are fulfilled:
\begin{align*}
  \sqrt{\chi}, \sqrt{\chi'}, (1-\chi^2)^{1/4} ,
  \sqrt{-\parb{(1-\chi^2)^{1/2} }'}\in C^\infty.
\end{align*} We define correspondingly $\chi_+=\chi$ and
$\chi_-=(1-\chi^2)^{1/2} $ and record that
\begin{align*}
  \chi_+^2+\chi_-^2=1\mand\sqrt{\chi_+}, \sqrt{\chi_+'}, \sqrt{\chi_-}, \sqrt{-\chi_-'}\in C^\infty.
\end{align*}

We denote the space of  bounded operators 
from one  (abstract) Banach space $X$ to another one $Y$ by $\vL(X,Y)$ 
and abbreviate $\mathcal L(X,X)=\mathcal L(X)$.  If $T$ is a self-adjoint  operator on a
Hilbert space $\vG$ and $\varphi\in \vG$, then
$\inp{T}_\varphi:=\inp{\varphi,T\varphi}$.

\subsubsection{More
  notation}\label{subsubsec:More 
  notation}

We shall use the notation 
$\inp{x}=\parb{1+\abs{x}^2}^{1/2}$ for  $x\in \mathbf X$, and  we
write $\hat x=x/\abs{x}$ for $x\in \mathbf X\setminus\set{0}$.
 Denote the standard \emph{weighted $L^2$ spaces} by 
$$
L_s^2=L_s^2(\mathbf X)=\inp{x}^{-s}L^2(\mathbf X)\ \ \text{for
}s\in\mathbb R ,\quad  L_{-\infty}^2=\bigcup_{s\in\R}L^2_s,\quad
L^2_\infty=\bigcap_{s\in\mathbb R}L_s^2.
$$

Let $T$ be an  operator on $\mathcal H=L^2(\mathbf X)$ such that
$T,T^*:L^2_\infty\to L^2_\infty$, and let $t\in\mathbb R$.  Then we
say that   $T$ is an {\emph{operator of order $t$}}, if 
 for each $s\in\mathbb R$  the restriction  $T_{|L^2_\infty}$ extends to
 an operator $T_s\in\vL(L^2_{s}, L^2_{s-t})$. Alternatively stated,
 \begin{subequations}
  \begin{equation}\label{eq:defOrder}
\|\inp{x}^{s-t}T\inp{x}^{-s}\psi\|\le C_s\|\psi\| \text{ for all }\psi\in
L^2_\infty.
\end{equation} 
 If   $T$ is of {order $t$}, we write 
\begin{equation}
  T=\vO(\inp{x}^t)\,\text{ or }\,T=\vO(r^t).
\label{eq:1712022}
\end{equation} The latter notation is motivated   by
\eqref{eq:compa00} stated below.
\end{subequations}

\subsection{Basic  $N$-body scattering   theory}\label{subsec:body preliminaries}
Under a rather weak condition (in particular weaker  than Condition
\ref{cond:smooth2wea3n12}) it is demonstrated in  \cite{AIIS},  that 
for any $s>1/2$
  \begin{align}\label{eq:BB^*a}
    \sup_{0<\epsilon\le 1}\norm{R(\lambda\pm\i
    		\epsilon)}_{\vL\parb{L^2_s,L^2_{-s}}}&<\infty \text{ with 
    		locally
    		uniform
    		bounds in 
    	}\lambda\in\R\setminus\vT_\p(H),\nonumber\\
      &\text{and there exist the limits}\\
                                                              R(\lambda\pm \i
                                                              0)=\lim_{\epsilon\to 0_+} &\,R(\lambda\pm\i
                                                              \epsilon)\in \vL\parb{L^2_s,L^2_{-s}};\quad \lambda\in\R\setminus\vT_\p(H).\nonumber
   \end{align}
\begin{subequations}
 These results follow from a Mourre estimate for the operator $A:=r^{1/2}Br^{1/2}$,
where
\begin{equation}\label{eq:Mourre}
   B=r^{-1/2}Ar^{-1/2}=\i[H,r]=-\i\sum_{j\leq d}\parb{\omega_j\partial_{x_j}+\partial_{x_j}\omega_j},\quad  \omega=\mathop{\mathrm{grad}} r   
\end{equation}  and  $r $ is   a certain  smooth positive
function  on $\mathbf X$  (taken from \cite{De}) fulfilling   the
following property, cf.  \cite[Section 5]{Sk1}:
 For all $k\in\N_0:=\N\cup\set{0}$ and $\beta\in \N^{d}_0$   there
  exists $C>0$ such that 
\begin{equation}\label{eq:compa00}
  \abs[\big]{\partial ^\beta (x\cdot \nabla)^k \parb{r(x)-\inp{x}}} \leq C\inp{x}^{-1}.
\end{equation}
 Fixing any real $\lambda\not\in \vT_{\p}(H)$ we can indeed find such  function
(possibly taken rescaled, meaning that we replace $r(x)$ by
$Rr(x/R)$ with  $R\geq 1$   large)  and  a small 
open neighbourhood $U_\lambda$ of $\lambda$, such that
\begin{equation}\label{eq:mourreFull}
  \forall \text{ real } f\in C_\c^\infty(U_\lambda):\quad{f}(H)
  \i[H,A]{f}(H)\geq
 4c\,{f}(H)^2;
\end{equation} here the constant $c>0$ and its size  depends on the number
\begin{equation}\label{eq:optimal}
  d(\lambda,H)=\dist\parb{\lambda ,\set{E\in\vT (H)\mid E<\lambda}}.
\end{equation} In fact $c$  can be taken arbitrarily smaller than
the number $d(\lambda,H)$ (upon correspondingly adjusting the scaling parameter $R$
and the  neighbourhood $U_\lambda$), however to  avoid  ambiguities
we take  concretely $c=\sigma^2$,
   $\sigma/\sigma_0=10/11$, where   $\sigma_0:=\sqrt{d(\lambda,H)}$,
   which suffices for the applications in this paper.   
\end{subequations}

This   estimate \eqref{eq:mourreFull} entails \eqref{eq:BB^*a} and
therefore (by the familiar Kato theory \cite{Ka})  also 
the local decay 
bound
\begin{align}
  \begin{split}
  \forall  f\in &C_\c^\infty(U_\lambda)\,\forall \psi\in \vH:\int^\infty_{-\infty} \,\norm{\inp{x}^{-\delta}\psi(t)}^2\,\d t\leq
  C_\delta \norm{\psi}^2;\\&\quad \psi(t)=\e^{-\i t H}f(H)\psi,\quad \delta>1/2.\label{eq:smoothBnd1}  
  \end{split}
\end{align}

There are related forward  (and backward) bounds of propagation. Thus   
for example, for any fixed $\lambda\not\in
 \vT_{\p}(H)$,
    \begin{equation*}
      R(\lambda+ \i0)\psi-\chi_+(  B/\sigma_0) R(\lambda+ \i0)\psi\in L^2_{-1/2}
      \text{ for all }\psi\in \bigcup_{s>1/2}L^2_s,
    \end{equation*} cf.   \cite [Subsection 5.2]{Sk1}.   In
    Appendix \ref{sec:Strong bounds}  we prove 
    the following  analogous  bound 
\begin{align}\label{eq:MicroB0}
  \begin{split}
   \forall  f\in &C_\c^\infty(U_\lambda)\,\forall \psi\in \vH:\\
      &\quad \lim_{t\to +\infty}\norm{\psi(t)-\chi_+(  B/\sigma_0)\psi(t)) }=0;\quad \psi(t)=\e^{-\i t H}f(H)\psi. 
  \end{split}
\end{align} 

In addition to the smoothness bound \eqref{eq:smoothBnd1} there is a
class of such  bounds
    to be crucial  for our proof of asymptotic completeness (as it is in \cite{Ya1}). These    bounds are   stated as  \eqref{eq:smoothBnd2}.
The  assertion \eqref{eq:MicroB0} will play  a somewhat   intermediate role.

The above assertions for $H$ are clearly 
also valid for the
Hamiltonians $H_a$, $a\in\vA'$, and there are
  similar  assertions  for
the sub-Hamiltonains $H^a$ upon replacing $(H,r)$ by analogous  pairs
$(H^a,r^a)$.

Parallel to \eqref{eq:wave_opq} and \eqref{eq:ACo} we define  
a channel as a triple 
$\alpha =(a,\lambda^\alpha, u^\alpha)$, where    $a\neq a_{0}$, $\lambda^\alpha\in \R$ and  
$(H^a-\lambda^\alpha)u^\alpha=0$ for a normalized $u^\alpha\in
\mathcal H^a$. The corresponding channel wave operators are given as
the strong limits
\begin{subequations}
 \begin{equation}\label{eq:wave_opq2}
  W_\alpha^{\pm}=\slim_{t\to \pm\infty}\e^{\i tH}J_\alpha\e^{-\i
   t( -\Delta_{x_a}+\lambda^\alpha )};\quad J_\alpha \varphi
=u^\alpha\otimes
\varphi.
 \end{equation} With the  existence of the channel wave operators
 given,   asymptotic completeness means that
 \begin{equation}\label{eq:ACo2}
   \Sigma_\alpha^\oplus\,\, R( W_\alpha^{\pm})=\vH_{\rm ac}(H).
 \end{equation} 
\end{subequations}

  Our  proof of completeness  can easily be adjusted to a proof of the
existence of the  channel wave operators. It is a general fact  that the existence of the channel 
wave operators  implies their orthogonality, see for example
\cite[Theorem XI.36]{RS}.
Consequently we shall focus
on proving the `onto-property' of \eqref{eq:ACo2}  only, in fact only for  the forward
direction $t\to +\infty$, i.e. for the `plus' case. We devote Section
\ref{sec:Asymptotic completeness} to this issue.

As the reader will see in Subsection
\ref{subsec:A partition of unity} we will use \eqref{eq:MicroB0} 
 to construct   a partition of unity for  the  given
scattering state as in \eqref{eq:locIni}. The partition observables, whose building blocks
$(M_a)_{a\in \vA'}$ 
will be introduced  in Section \ref{sec: Yafaev type
  constructions}, will thanks to the local smoothness 
bounds 
\eqref{eq:smoothBnd1} and \eqref{eq:smoothBnd2} (and their analogous
versions for  the operators $H_a$) have controllable
commutator properties allowing us to construct limits as  in
\eqref{eq:parta}, indeed yielding  asymptotic completeness completeness.

\section{Yafaev's  constructions}\label{sec: Yafaev type
  constructions}

In \cite{Ya1} various  real functions $m_a\in C^\infty(\mathbf X)$,
$a\in\vA'$,  and $m_{a_{0}}\in C^\infty(\mathbf X)$  are 
constructed. They  are  homogeneous of degree $1$ for $\abs{x}> 5/6$,
and in addition $ m_{a_0}$ is (locally) convex in that region. These
functions are constructed as depending  on a  small parameter
$\epsilon>0$, and once given, they  are  used in the constructions
\begin{equation}\label{eq:M_a0}
 M_a=2\Re(w_a\cdot p)=-\i\sum_{j\leq d}\parb{(w_a)_j\partial_{x_j}+\partial_{x_j}(w_a)_j};\quad  w_a=\mathop{\mathrm{grad}} m_a.   
\end{equation}
The operators $M_a$, $a\in\vA'$,  play the role of  `channel
localization operators', while  operators on the form $M_{a_0}$
(parametrized by $\epsilon>0$) mainly are 
used  as  technical quantities  controlling  commutators of  
the Hamiltonian and the channel
localization operators. 
 In Section \ref{sec:Asymptotic completeness} we implement the
same ideas  for a proof of asymptotic completeness, although in a  different way than Yafaev did.

We  consider various conical
subsets of  $\mathbf X\setminus\set{0}$, defined for 
  $a\in\vA'$ and $\varepsilon,\delta\in (0,1)$,
\begin{align}\label{eq:primes}
  \begin{split}
    \mathbf X'_a&={\mathbf X_a }\setminus\cup_{{b\not\leq a,\,b\in
                     \vA}\,
                 }\mathbf X_b,\\
\mathbf X_a(\varepsilon)&=\set{x\in
  \bX\mid\abs{x_a}>(1-\varepsilon)\abs{x}},\\\quad
&\quad\quad\overline{\mathbf X_a(\varepsilon)}=\set{x\in \bX\setminus\set{0}\mid\abs{x_a}\geq(1-\varepsilon)\abs{x}},\\
\mathbf Y_a(\varepsilon)&=\parb{\mathbf X \setminus\set{0}}\setminus\cup_{{b\not\leq a,\,b\in \vA'}
}\,\mathbf X_b(\varepsilon),\\
\mathbf Z_a(\delta)&=\mathbf X_a(\delta)\setminus\cup_{b\not\leq a,\,b\in \vA'}
\,\overline{\mathbf X_b(3\delta^{1/d_a})};\quad\quad d_a=\dim
\mathbf X_a.
\end{split}
\end{align}
The structure of the sets $\mathbf X_a(\varepsilon)$, $\mathbf
Y_a(\varepsilon)$ and $\mathbf Z_a(\delta)$ is
${\R_+ V}$, where  $V$ is a
subset of the unit sphere $\mathbf{S}^{a_0}\subset\mathbf X$. For $\mathbf
X_a(\varepsilon)$ and $\mathbf Z_a(\delta)$ the set $V$ 
 is relatively open, while for  $\mathbf Y_a(\varepsilon)$ the  set is  compact.
We also record that $\mathbf X_a\setminus\set{0}\subseteq\mathbf
  X_a(\varepsilon)\subseteq \set{\abs{x^{a}}< \sqrt{2\ \varepsilon}\abs{x}}$,
and that $\mathbf Z_a(\delta)\subseteq \mathbf X_a(\delta)\cap\mathbf Y_a(3\delta^{1/d_a})$. The following inclusion is less obvious (we include
  its proof  from \cite{Sk2}).

\begin{lemma}\label{lemma:cover} For  any $a\in\vA'$ and $\varepsilon\in (0,1)$
\begin{equation*}
   \mathbf Y_a(\varepsilon)\subseteq\cup_{b\leq a}  \,\mathbf X'_b.
\end{equation*} 
  \end{lemma}
  \begin{proof}
    For any $x\in 
\mathbf Y_a(\varepsilon)$ we introduce $b(x)\leq a$ by  $\mathbf X_{b(x)}=\cap_{b\leq
a,\,x\in \mathbf X_b} \,\mathbf X_b$ and then check that $x\in \mathbf
X'_{b(x)}$: If not, $x\in \mathbf X_b \cap \mathbf X_{b(x)}$ for some
$b$ with $b(x)\lneq b\leq a$. Consequently 
    $\mathbf X_b=\mathbf X_b\cap \mathbf X_{b(x)}=\mathbf X_{b(x)}$,  contradicting that
  $b\neq b(x)$. 
  \end{proof}

The Yafaev constructions involve families of  cones $\mathbf X_a(\varepsilon)$ roughly with    `width'
$\varepsilon\approx\epsilon^{d_a}$ for a 
parameter $\epsilon>0$. This parameter is taken sufficiently small   as primarily 
determined  by the following
geometric property: 
There exists $C>0$ such that for all $a,b \in \vA'$ and $x\in\mathbf X$
\begin{equation}\label{eq:3.2}
  \abs{x^c}\leq C\parb{\abs{x^a}+\abs{x^b}};\quad c=a\vee b.
  \end{equation}
We note that \eqref{eq:3.2} is  also vital in Graf's constructions
\cite{Gra, De}.

\subsection{Yafaev's  homogeneous functions}\label{subsec: Homogeneous Yafaev type functions} 

For small $\epsilon>0$ and any $a\in \vA'$ we define
\begin{equation*}
  \varepsilon^a_k= k\epsilon^{d_a},\quad k=1,2,3,
\end{equation*} and view the numbers  $\varepsilon_a$ in the
interval 
$(\varepsilon^a_2, \varepsilon^a_3)$  as 
`admissible'. Alternatively, viewed as a free parameter, $\varepsilon_a$ is
called admissible if 
\begin{equation}\label{eq:admis}
  \varepsilon^a_2=2\epsilon^{d_a}<\varepsilon_a<3\epsilon^{d_a}=\varepsilon^a_3.
\end{equation} Using an  arbitrary
ordering of $\vA'$ we introduce for these  parameters  the `admissible' vector
\begin{equation*}
  \bar\varepsilon=(\varepsilon_{a_1}, \dots
\varepsilon_{a_n}),\quad n= \# \vA',
\end{equation*} and denote by $\d \bar\varepsilon$ the corresponding
Lebesgue measure. 

Let for $\varepsilon>0$ and $a\in \vA'$
\begin{equation*}
  h_{a,\varepsilon}(x)=(1+\varepsilon)\abs{x_a};\quad
  x\in \mathbf X\setminus \set{0}.
\end{equation*} 
 Letting  $\Theta= 1_{[0,\infty)}$,  we define for  any $a\in \vA'$ and any admissible vector $\bar\varepsilon$
  \begin{equation*}
    m_a(x, \bar \varepsilon)= h_{a,\varepsilon_a}(x)\Theta\parb{
      h_{a,\varepsilon_a}(x)-\max_{\vA' \ni c\neq a } h_{c,\varepsilon_c}(x)};\quad
  x\in\mathbf X\setminus \set{0}.
  \end{equation*} 
Introducing  for any  $c\in \vA'$ a
non-negative function $\varphi_c\in C^\infty_\c(\R)$ with
\begin{equation*}
  \supp \varphi_c\subseteq (\varepsilon^c_2,\varepsilon^c_3)
\mand \int_\R \varphi_c(\varepsilon)\,\d \varepsilon=1,
\end{equation*} we  average  the functions $ m_a(x,
\bar \varepsilon)$,     $a\in \vA'$, over $\bar
\varepsilon$ as    
\begin{equation*}
  m_a(x)=\int_{\R^n} m_a(x, \bar \varepsilon)\,\prod_{c\in \vA'}\, \varphi_c(\varepsilon_c)\,\d \bar\varepsilon;\quad x\in\mathbf X\setminus \set{0}.
\end{equation*} 
\begin{lemma}\label{lemma:ma1}
For any $a\in\vA'$ the  function 
 $m_a:\mathbf X\setminus \set{0}\to \R$ fulfills  the following
properties for any sufficiently small  $\epsilon>0$ and any $b\in\vA'$: 
\begin{enumerate}[1)]
\item\label{item:10a} $m_a$ is homogeneous of degree $1$.
\item\label{item:11a} $m_a\in C^\infty (\bX\setminus\set{0})$.
\item\label{item:12a} If $b\leq a$  and  $x\in \mathbf X_b(\varepsilon^b_1)$, then
  $m_a(x)=m_a(x_b)$.
\item\label{item:13a} If ${b\not\leq a}$ and $x\in \mathbf
  X_b(\varepsilon^b_1)$, then   $m_a(x)=0$.
\item\label{item:14a} If $a\neq {a_{\min}}$,  $x\neq 0$  and $\abs{x^{a}}\geq \sqrt{2\ \varepsilon^a_3}\abs{x}$, then   $m_a(x)=0$.
\end{enumerate} 
\end{lemma}

Following \cite{Ya1} we introduce for  $ x\in\mathbf X\setminus \set{0}$
and admissible vectors $\bar\varepsilon$
\begin{equation*}
   m_{a_0}(x,\bar\varepsilon)=\max_{a\in \vA'} h_{a,\varepsilon_a}(x),  
\end{equation*} and let then
\begin{equation*}
  m_{a_0}(x)=\int_{\R^n} m_{a_0}(x, \bar \varepsilon)\,\prod_{c\in
    \vA'}\, \varphi_c(\varepsilon_c)\,\d \bar\varepsilon;\quad x\neq 0.
\end{equation*}
\begin{lemma}[\cite{Ya1, Sk2}]\label{lemma:m1} The  function 
 $m_{a_0}:\mathbf X\setminus \set{0}\to \R$ fulfills  the following
properties for any sufficiently small  $\epsilon>0$: 
\begin{enumerate}[i)]
\item\label{item:10b} $m_{a_0}$ is convex and homogeneous of degree $1$.
\item\label{item:11b} $m_{a_0}\in C^\infty (\bX\setminus \set{0})$.
\item\label{item:12b} If $b\in\vA'$  and  $x\in \mathbf X_b(\varepsilon^b_1)$, then
  $m_{a_0}(x)=m_{a_0}(x_b)$.
\item\label{item:9b} $m_{a_0}=\sum_{a\in \vA'}\,m_a$  (with all functions defined for the same $\epsilon$).
\item\label{item:13b} Let $a\in \vA'$ and suppose $x\in\mathbf
  X_a(\varepsilon^a_1)$ obeys that for all  $b\in \vA'$
  with  $b\gneq
  a $ the vector $x\not\in\mathbf
  X_b(\varepsilon^b_3)$. Then 
  \begin{equation*}
    m_{a_0}(x)=m_{a}(x)=C_a
    \abs{x_a};\quad C_a=\int_\R (1+\varepsilon)\,\varphi_a(\varepsilon)\,\d \varepsilon.
  \end{equation*}
\item\label{item:15b} The derivative of $m_{a_0}\in C^\infty
  (\bX\setminus \set{0})$ obeys 
\begin{align*}
   \quad \quad\nabla \parb{m_{a_0}(x)-\abs{x}}&\\=\sum_{a\in \vA'}\,\int
   \,
 \d \varepsilon_a&
   \nabla \parb{{h_{a,\varepsilon_a}(x)}-\abs{x}}\,\varphi_a(\varepsilon_a)\\  &\prod_{\vA'\ni b\neq a    }\,\parbb{\int\,\Theta\parb{ h_{a,\varepsilon_a}(x)-h_{b,\varepsilon_b}(x)}\varphi_b(\varepsilon_b)\, \d\varepsilon_b}. 
\end{align*} In particular there exists $C>0$ (being independent of the
parameter  $\epsilon$) such that for all $x\in
\bX\setminus \set{0}$ 
\begin{equation}
  \label{eq:compa}
  \abs{\nabla \parb{m_{a_0}(x)-\abs{x}}}\leq C\sqrt{\epsilon}.
\end{equation}
\end{enumerate} 
\end{lemma}

\subsection{Yafaev's  observables and their commutator  properties}\label{subsec: Geometric considerations for a0=a{max}}
 We introduce smooth modifications of the functions $m_a$ from Lemma \ref{lemma:ma1}
and $m_{a_0}$ from Lemma \ref{lemma:m1}  by
multiplying them by a suitable factor, say specifically by the factor
$\chi_+(2|x|)$. We adapt these modifications and will (slightly
abusively) use  the same notation $m_a$ and $m_{a_0}$ for the smoothed out
versions.  We may then consider the corresponding first order
operators $M_a$ and $M_{a_0}$ from \eqref{eq:M_a0} as realized as self-adjoint
operators. 

In our application in 
Section \ref{sec:Asymptotic completeness}  the above Yafaev  observables  will be used to construct a suitable
 partition of unity based on which asymptotic completeness
is established. In the present subsection we provide
technical details on how to control their commutator with the
Hamiltonian $H$. Although not being treated explicitly, similar
arguments work for the  Hamiltonians $H_a$, $a\in\vA'$.

 Thanks to   Lemma \ref{lemma:cover}  we can for  any $a\in\vA'$ and  any  $\varepsilon, \delta_0\in (0,1)$  write
  \begin{equation*}
    \mathbf Y_a(\varepsilon)\subseteq\cup_{b\leq a} \cup_{\delta\in (0,\delta_0]} \,\mathbf Z_b(\delta).
  \end{equation*} 
  By compactness this   leads for
  any fixed $\varepsilon,\delta_0\in(0,1)$ to the existence of a finite
  covering
\begin{equation} \label{eq:deltaNBHb}
     \mathbf Y_a(\varepsilon)\subseteq\cup_{j\leq J} \,\,\mathbf Z_{b_j}(\delta_j),
  \end{equation} where  $\delta_1,\dots, \delta_J\in
(0,\delta_0]$ and  $b_1,\dots, b_J\leq a$; $J=J(a)$. Here and in the
following we prefer to  suppress
 the dependence of quantities on the given $a\in\vA'$.

 We recall   that  the
functions $m_a$, $a\in \vA'$, depend on a
positive parameter $\epsilon$ and  by Lemma  \ref{lemma:ma1}
\ref{item:13a}  fulfil 
\begin{equation}\label{eq:suppa}
  \supp m_a\subseteq 
\mathbf Y_a(\epsilon^{d});\quad d=\dim
\mathbf X.
\end{equation}
    Having  Lemmas \ref{lemma:ma1} and \ref{lemma:m1} at our
    disposal we also observe   that the (small) positive  parameter
    $\epsilon$ appears independently for  the two lemmas.  Let us 
    choose and fix the \emph{same}   small $\epsilon$ for the  lemmas,
    yielding in this way  the whole family  $(M_a)_{a\in \vA}$
    uniquely defined. Thanks to Lemma
\ref{lemma:m1} \ref{item:9b}  we can  then record that 
     $M_{a_0}=\sum_{a\in \vA'}\,M_a$,  which will  be  used  in
     Section \ref{sec:Asymptotic completeness} (more precisely  possibly with
     $\epsilon$ taken smaller, see Lemma \ref{lem:compBogB}).
\subsubsection{Controlling the commutators $\i[
H,M_a]$}\label{subsubsec: Controlling commutators}
We need addditional   applications of Lemma \ref{lemma:m1}, more
precisely for certain inputs $\epsilon_1,\dots,\epsilon_J\leq
\epsilon$ (with $J=J(a)$, 
$a\in\vA'$)
     defined  as follows. 
We apply \eqref{eq:deltaNBHb} with
    $\varepsilon=\delta_0=\epsilon^{d}$
    (cf. \eqref{eq:suppa}). Since $\delta_j\leq \delta_0$, we can
    introduce positive $\epsilon_1,\dots,\epsilon_J\leq \epsilon$ by  the
    requirement $\epsilon_j^{d_{b_j}}=\delta_j$. The 
    new inputs $\epsilon=\epsilon_j$ in Lemma \ref{lemma:m1}  yield
    corresponding   functions, say denoted   $m_j$. In
    particular in the region $\bZ_{b_j}(\delta_j)$ the function $m_a$
    from Lemma \ref{lemma:ma1} only
    depends on $x_{b_j}$ (thanks to   Lemma
    \ref{lemma:ma1} \ref{item:12a} and the  property
    $\bX_{b_j}(\delta_j)\subseteq \bX_{b_j}(\epsilon^{d_{b_j}})$),
    while (thanks to Lemma \ref{lemma:m1} \ref{item:13b})
    \begin{equation}\label{eq:good0}
      m_j(x)=C_j\abs{x_{b_j}},\quad C_j=\int_\R (1+\varepsilon)\,\varphi_{b_j}(\varepsilon)\,\d
      \varepsilon.
    \end{equation} 
Motivated by \eqref{eq:good0} we introduce   
the vector-valued first order operators 
\begin{equation}\label{eq:Ffield}
  G_b=\chi_+(2\abs{x_b})\abs{x_b}^{-1/2} P(x_b)\cdot
  p_b,\text{ where } P(x)=I-\abs{x}^{-2}\ket{ x}\bra{ x}.
\end{equation}  
 We  choose for the considered   $a\in\vA'$ 
   a quadratic partition $\hat\xi_1,\dots, \hat\xi_J\in
C^\infty(\mathbf{S}^{a_0})$ (viz $\sum_j \,{\hat\xi_j}^2=1$)
subordinate to the covering \eqref{eq:deltaNBHb} (recalling the discussion before
    Lemma \ref{lemma:cover}). Then
we can write, using  the support properties \eqref{eq:suppa} and \eqref{eq:deltaNBHb}, 
\begin{align*}
  m_a(x)=\sum_{j\leq J}\,\,m_{a,j}(x);\quad m_{a,j}(x)=\hat\xi^2_j( \hat x) m_a(x),\quad \hat x:=x/\abs{x},
\end{align*} and from the previous discussion it follows  that 
\begin{align*}
  \chi^2_+(\abs{x})m_{a,j}(x)=\xi^2_j
  m_{a}(x_{b_j}),\quad \xi_j=\xi_j(x):=\hat\xi_j( \hat x)\chi_+(\abs{x}),
\end{align*} as well as 
\begin{align}\label{eq:Hes0}
  \begin{split}
 p\cdot\parb{\chi^2_+(\abs{x})\nabla^2m_a(x)} p&=\sum_{j\leq J}\,\,
                                 p\cdot\parb{\xi^2_j\nabla^2m_a(x_{b_j})}p\\&
=\sum_{j\leq J}\, \,G^*_{b_j}\parb{\xi^2_j\vG_j}G_{b_j};\quad 
                                                                  \vG_j=\vG_j(x_{b_j})\text{ 
                                                                  bounded}.   
  \end{split}
\end{align} 
\begin{subequations}
We note that   used with appropriate energy-localization,  the left-hand side and 
$\tfrac 14 \i[
p^2,M_a]$ coincide
up to order $-3$.
By the convexity property of $m_j$ and \eqref{eq:good0}
\begin{align}\label{eq:Hes_est}
   G^*_{b_j}\xi_j^2
  G_{bj}\leq   p\cdot \parb{\chi^2_+(\abs{x})\nabla^2m_j(x)} p.
\end{align} 
Here the right-hand side  is the `leading term' of  the commutator
$\tfrac 14 \i[
p^2,M_j]$, where $M_j$ is given by \eqref{eq:M_a0} for  the
modification of 
$m_j$ given by the  function $\chi_+(2|x|)m_j(x)$. More precisely,
cf. \cite[Lemma 6.13]{Sk1} and recalling  \eqref{eq:1712022}, for any real $f\in C_\c^\infty(\R)$
\begin{align}\label{eq:1712022apll}
	\begin{split}
		f(H)\parbb{ p\cdot \parb{\chi^2_+(\abs{x})\nabla^2m_j(x)} p-\tfrac 14 \i[
H,&\chi_+(\abs{x})M_j\chi_+(\abs{x})] }f(H) \\&=\vO(\inp{x}^{-1-2\mu}).
\end{split}
\end{align}  
\end{subequations} We note that a basic  application of
\eqref{eq:Hes_est} and \eqref{eq:1712022apll}  comes from combining these
features  with \eqref{eq:smoothBnd1} in a commutator argument based on the `propagation observable'
\begin{equation*}
  \Psi=\Psi_j=\tfrac 14 f_1(H)\chi_+(\abs{x})M_j\chi_+(\abs{x})f_1(H),
\end{equation*} where 
$ f_1$ is any  narrowly supported  support
  function obeying  $f_1=1$ in a
neighbourhood of a given real  $\lambda\not\in \vT_{\p}(H)$, cf. \cite
[Subsection 7.1]{Sk1}. This argument, involving  computation  and
bounding of $\i[H,\Psi]$, leads  to the following smoothness
bounds:

With  $j\leq J=J(a)$, $\xi_j=\xi_j(x)$ and $G_{b_j}$ given as in
\eqref{eq:Hes0}, it follows that 
\begin{align}
  \begin{split}
   \forall  f\in &C_\c^\infty(U_\lambda)\,\forall \psi\in \vH:\quad \int^\infty_{-\infty} \,\norm[\big]{Q(a,j)\psi(t)}^2\,\d t\leq
  C \norm{\psi}^2;\\&\quad  Q(a,j)=\xi_jG_{b_j},\quad\psi(t)=\e^{-\i t H}f(H)\psi.\label{eq:smoothBnd2} 
  \end{split}
\end{align}
We will in Section \ref{sec:Asymptotic completeness} use
\eqref{eq:Hes0} and \eqref{eq:smoothBnd2} in combination with
\eqref{eq:smoothBnd1} to treat appearing expressions with factors of
the  commutator
$\i[
H,M_a]$; $a\in\vA'$.

\subsubsection{Localization and more calculus properties}\label{subsubsec: Geometric
  considerations cont} 
Choose for each $a\in\vA'$ a function  $\hat\xi_a\in
C^\infty(\mathbf{S}^{a_0})$  such that $\hat \xi_a=1$ in
$\mathbf{S}^{a_0}\cap\mathbf  Y_a(\epsilon^{d})$ and $\hat\xi_a=0$ on
$\mathbf{S}^{a_0}\setminus  \mathbf Y_a(\epsilon^{d}/2)$, and  let  $\xi_a=\xi_a(x)=\hat\xi_a(\hat
x)\chi_+(4\abs{x})$.  It then follows from Lemma
\ref{lemma:ma1} \ref{item:13a} that the operator $M_a$ fulfills   
\begin{align}\label{eq:partM}
  M_a= M_a\xi_a= \xi_aM_a,
\end{align} which is  useful for replacing $H$ by
$H_a$ (or vice versa) in the presence of a factor $M_a$. For our
application  in 
\eqref{eq:bndInterclust} commutation
applies smoothly. A  trivial application of
\eqref{82a0} yields the related commutation result  
\begin{equation}\label{eq:rCalc1}
  \forall f\in C_\c^\infty(\R):\quad
  \i[f(H),T]=\vO(\inp{x}^{-1})\inp{p}^{-1};\quad T=\xi_a.
\end{equation}
\begin{subequations}

There are other   such `smooth' commutation results, for example the
bound 
\eqref{eq:rCalc1} also holds  for $T=M_a$ as well as  for
$T=\abs{x}^{1/2}Q(a,j)=\abs{x}^{1/2}\xi_j(x)G_{b_j}$. For the latter
results one can actually   based on  
 \eqref{82a0}  easily expand to first order as 
\begin{align}\label{eq:rCalc2}
  \forall f\in C_\c^\infty(\R):\quad
  \i[f(H),M_a]&=\i[p^2,M_a]f'(H)+\vO(\inp{x}^{-1-2\mu})\inp{p}^{-1},\\
\forall f\in C_\c^\infty(\R):\,
  \i[f(H),Q(a,j)]&=\i[p^2,Q(a,j)]f'(H)+\vO(\inp{x}^{-3/2-2\mu})\inp{p}^{-1},
\label{eq:rCalc3}
\end{align} cf. \cite[Lemma 6.13]{Sk1}.  We shall in  Section
\ref{sec:Asymptotic completeness} use 
\eqref{eq:rCalc2} and \eqref{eq:rCalc3} as well as the  adjoint
form of \eqref{eq:rCalc2}. 
\end{subequations}

\section{Asymptotic completeness}\label{sec:Asymptotic completeness}
Taking for convenience  the  existence of the channel wave operators of
\eqref{eq:wave_opq2} for
granted (it  can be shown by a similar reasoning as the one  presented below),  we prove in this section the
following main result.
\begin{thm}\label{thm:asympt-compl} The channel wave operators
  $W_\alpha^+$ of \eqref{eq:wave_opq2}  are complete in the sense
  that their ranges span the
 absolutely continuous subspace  of $H$  
  \begin{equation*}
   \Sigma_\alpha^\oplus\,\, R( W_\alpha^+)=\vH_{\rm ac}(H).
 \end{equation*} 
\end{thm}

 By an elementary  density and covering argument, 
combined with  an
   induction argument, Theorem \ref{thm:asympt-compl} is a 
 consequence of
 the following result.
 \begin{subequations}
 \begin{lemma}\label{lemma:asympt-compl} Let
   $\lambda\in\R\setminus\vT_\p(H)$. There exist $H$-bounded operators
    $(S_a)_{a\in \vA'}$ and an 
open neighbourhood $U$  of $\lambda$ such that for any  $f\in
C_\c^\infty(U)$,
    any  support
  function
$f_1\succ f$, also supported in $U$, the following properties hold for
$\psi(t)=\e^{-\i t H}f(H)\psi$, $\psi\in \vH$. 
 \begin{equation}\label{eq:part}
    \norm[\Big]{\psi(t) - \sum_{a\in \vA'}\,f_1(H_a)S_a \psi(t)}\to 0\text{ for
    }t\to +\infty.
  \end{equation} 
   \begin{equation}\label{eq:part2}
  \text{ There exist the limits}: \quad \psi_a=\lim_{t\to +\infty}\,\e^{\i tH_a}f_1(H_a)S_a\psi(t); \quad a\in \vA' .
\end{equation} 
\end{lemma}  
 \end{subequations} In fact, given Lemma \ref{lemma:asympt-compl},  it
 follows that 
\begin{equation*}
    \norm[\Big]{\psi(t) - \sum_{a\in \vA'}\,\e^{-\i tH_a}\psi_a}\to 0\text{ for
    }t\to +\infty,
  \end{equation*} and by introducing a suitable decomposition in the summation,
  \begin{equation*}
    I=f^a(H^a)+\parb{I-f^a(H^a)},\quad f^a\in C_\c^\infty(\R\setminus{\vT_\p(H^a))},
  \end{equation*}
  the contribution from the first term can
  be treated by induction while the second  one  agrees (up to an arbitrary
  error) with
  the completeness assertion, cf. \cite{SS}.

\subsection{Partition of unity}\label{subsec:A partition of unity}

Let  in this subsection $\lambda\notin \vT_\p(H)$,  $\psi\in \vH$ and $f\in
C_\c^\infty(U_\lambda)$ be given as in 
\eqref{eq:smoothBnd1} and \eqref{eq:MicroB0}. We shall
 localize the state $\psi(t)=\e^{\i t H}f(H)\psi$ for large $t$. This will be in terms of an `effective 
  partition of unity' $I\approx \Sigma_{a\in \vA'}\,S_a$, meaning more
  precisely that \eqref{eq:part} holds, taking here and below  $U=U_\lambda$. We  show
  \eqref{eq:part} with the observables $S_a$ as   specified in
  \eqref{eq:gDef} (given in terms of a small parameter $\epsilon$), leaving the  completion of  the proof of Lemma
  \ref{lemma:asympt-compl}  to  Subsection \ref{subsec:Proof of Lemma ref{lemma:asympt-compl}}.
The localization bound \eqref{eq:MicroB0} (involving the  operator
$\chi_+(B/\sigma_0)$) will play  an intermediate role  only.  

Recall that the
  construction of $(M_a)_{a\in \vA}$ in Subsection
  \ref{subsec: Geometric considerations for a0=a{max}} is based on an
  arbitrary  
small parameter
  $\epsilon>0$. We need the following  elementary result from \cite{Sk2},
  recalling  our notation \eqref{eq:1712022} and for the proof also the notation $\inp{T}_{\varphi}=\inp{\varphi,T\varphi}$.
\begin{lemma}\label{lem:compBogB}
 Let the   parameter
  $\epsilon$ in the construction of $M_{a_0}$  be sufficiently small
  (depending on the fixed $\lambda\notin \vT_\p(H)$),
  and let $f_1\in C_\c^\infty(\R)$ be given. Then 
  \begin{equation*}
    \chi_-(2M_{a_0}/\sigma_0)\chi_+(B/\sigma_0)f_1(H)=\vO(r^{-1/2}).
  \end{equation*}
\end{lemma}
\begin{proof}  Take any 
$f_2\succ f_1$. 
Introducing 
  \begin{equation*}
    S=f_2(H)\chi_-(2M_{a_0}/\sigma_0)
  \chi_+(B/\sigma_0)\inp{x}^{1/2}f_1(H),
  \end{equation*} it suffices  to be show that
  $S=\vO(\inp{x}^{0})$. Below we tacitly use various commutation based on \eqref{82a0}.
  A  relevant and  rather detailed 
account  is given in 
  \cite[Section 6]{Sk1} (note in particular that it is not completely
  obvious how to commute the factors of $\chi_-(2M_{a_0}\sigma_0)$ and
  $\chi_+(B/\sigma_0)$ using \eqref{82a0}; the ample energy-localization is
  conveniently used, see \cite[Remark 6.19]{Sk1}).

  We estimate   by   repeated
commutation:  For any $\brp\in
  L^2_\infty\subseteq \vH$
  \begin{align} \label{eq:commsmall0}
  \begin{split}
\tfrac{\sigma_0}6 \norm {{S\brp}}^2
&\leq \inp{\sigma_0 I-M_{a_0}
  }_{S\brp}+C_1\norm{\brp}^2 \\
&\leq \inp{(\sigma_0+C\sqrt{\epsilon}) I-B
  }_{{S\brp}}+C_2\norm{\brp}^2 \quad \quad \quad (\text{by
  }\eqref{eq:compa00} \mand \eqref{eq:compa})\\&
\leq \inp{(\sigma_0+C\sqrt{\epsilon}
  -\tfrac 43\sigma_0) I}_{{S\brp}}+C_3\norm{\brp}^2 
\\&\leq C_3\norm{\brp}^2\quad\quad \quad\quad\quad \quad\quad\quad \quad\quad\quad\quad \quad (\text{for }3C\sqrt{\epsilon}\leq \sigma_0). 
 \end{split} 
\end{align}

By repeating  the estimation
\eqref{eq:commsmall0} with $S\brp$ replaced by 
$\inp{x}^{s}S\inp{x}^{-s}\brp$, $s\in \R\setminus\set{0}$,  we
conclude  the
 estimate \eqref{eq:defOrder} with $t=0$ as wanted. Note that the
 above 
  constant $C$ works in this case also.
\end{proof}

Let us  for given `orbits'
 $\phi_1(t),\phi_2(t)\in \vH$ write 
$\phi_1(t) \simeq \phi_2(t)$   to signify that  $\norm{\phi_1(t) - \phi_2(t)}\to 0$
for $t\to +\infty $. The assertions \eqref{eq:MicroB0} and Lemma \ref{lem:compBogB}  should
allow us 
initially to localize (for any fixed small 
  $\epsilon>0$) as 
\begin{equation*}
  \psi(t)\simeq \chi_+(2M_{a_0}/\sigma_0)\chi_+(
  B/\sigma_0)\psi(t);\quad \psi(t)=\e^{-\i t H}f(H)\psi.
\end{equation*} We need the following version of this localization. 
\begin{lemma}\label{lemma:phase-space-part1} Let $g_+\in C^\infty (\R)$
  be given by $g_+(s)=\chi_+(2s/\sigma_0)/s$, $s\in \R$,  and let
  \begin{equation}\label{eq:gDef}
   S_a=g_+(M_{a_0})M_a;\quad a\in \vA'.
  \end{equation}
Then for any    support
  function
$f_1\succ f$, also supported in $U_\lambda$, 
  \begin{align*}
\psi(t)
  &\simeq f_1(H)\chi_+(2M_{a_0}/\sigma_0)\psi(t)\\ &\quad\quad\quad\quad=\sum_{a\in \vA'}\,f_1(H)S_a\psi(t)
\simeq \sum_{a\in \vA'}\,f_1(H_a)S_a\psi(t).
\end{align*}
\end{lemma}
\begin{proof} {\bf I}.  The middle    assertion is obvious from  Lemma
\ref{lemma:m1} \ref{item:9b}.

  	{\bf II}. For the first assertion it suffices to show that 
        \begin{equation*}
          L:=\lim_{t\to \infty}\,\norm{\chi_-(2M_{a_0}/\sigma_0)\psi(t)}^2=0.
        \end{equation*}

 First we show that the limit $L$ exists by 
  differentiating
$\norm{\chi_-(2M_{a_0}/\sigma_0)\psi(t)}^2$ with respect to
$t$ and integrating. We claim that indeed the derivative in integrable
thanks to 
\eqref{eq:smoothBnd1} and \eqref{eq:smoothBnd2} with  $\psi(t)$ replaced
with
$\psi_1(t):=\e^{-\i t H}f_1(H)\psi$. To see this we apply   the calculus result \cite[Lemma 6.12]{Sk1}
with $ B$  replaced by $ M_{a_0}$, hence letting 
$\theta(M_{a_0})=\sqrt{(-\chi^2_-)'(2M_{a_0}/\sigma_0)}$, 
$\brf_1(H)=Hf_1(H)$ and  $\bD_1 T=\i[\brf_1(H),T]$, it follows that 
\begin{align}\label{eq:deM}
\begin{split}
  f(H)^*\parb{ \bD_1 \chi^2_-(2M_{a_0}/\sigma_0)}&f(H)\\
= -\tfrac 2{\sigma_0}\theta(M_{a_0})&f(H)^*\i[H,M_{a_0}]f(H)\theta(M_{a_0})+
   \vO(r^{-2\mu-1}).
\end{split}
\end{align} 
We substitute \eqref{eq:deM}  in the inner product
$\inp{\bD_1\chi^2_-(2M_{a_0}/\sigma_0)}_{f(H)\psi_1(t)}$ and then in
turn substitute $\i[H,M_{a_0}]=\sum_{a\in \vA'}\,\i[H,M_a]$. Thanks to  
\eqref{eq:smoothBnd1} it
suffices  for each $a\in \vA'$ 
to consider  the contribution from \eqref{eq:Hes0}, cf. \cite[Lemma
6.13]{Sk1}. This  expression is treated by first pulling  
the appearing (components of) factors of
\begin{equation}\label{eq:Qdef}
  Q= Q(a,j)=\xi_jG_{b_j} \,\mand\,
Q^*
\end{equation}
to the right and  to the left in the inner product, amounting  to
compute the commutator 
 $[Qf(H),\theta(M_{a_0})]$ with $\theta(M_{a_0})$ represented by
 \eqref{82a0} using the  adjoint
form of \eqref{eq:rCalc2} to bound 
 \begin{equation*}
   [Qf(H),M_{a_0}]=[Q,M_{a_0}]f(H)+Q[f(H),M_{a_0}]=\vO(r^{-3/2})
 \end{equation*}
   and then estimate the errors  by  \eqref{eq:smoothBnd1}. Placed
next to the factors of $\psi_1(t)$ as $Qf(H)\psi_1(t)$,     the desired
integrability finally follows from \eqref{eq:smoothBnd2}. 
Hence indeed $L$ exists.

It remains to show that  $L=0$. We calculate using \eqref{eq:MicroB0} and Lemma \ref{lem:compBogB} 
\begin{align*}
  L&=\lim_{t\to
     \infty}\,\norm{\chi_-(2M_{a_0}/\sigma_0)\chi_+(B/\sigma_0)f_1(H)\psi(t)}^2\\
&= \lim_{t\to \infty}\,\norm[\big]{T\inp{x}^{-1/2}\psi(t)}^2;\quad  T=\vO(r^0).
\end{align*}
 By \cas and  the local decay estimate  \eqref{eq:smoothBnd1} 
 \begin{equation*}
   \int_0^\infty\,\norm[\big]{\inp{x}^{-1/2}\psi(t)}^4\,\d t\leq C
   \int_0^\infty\,\norm[\big]{\inp{x}^{-1}\psi(t)}^2\,\d t < \infty.
 \end{equation*} In particular it follows that $\norm{\inp{x}^{-1/2}\psi(t_n)}\to 0$
 along some sequence $t_n\to \infty$, yielding the conclusion   $L=0$, 
  and therefore also 
 the first   assertion.

	{\bf III}. We prove the  third    assertion. By using
        \eqref{82a0} and   \eqref{eq:partM}, it follows
        that  
\begin{align}
  \begin{split}
 \parb{f_1(H_a)-f_1(H)}S_af_1(H)&\\=\sum_{b\nleq a}\int _{\C}(H_a
  -z)^{-1}&V_b\Big [\big [(H -z)^{-1}g_+(M_{a_0}),\xi_a\big ],\xi_a\Big ] M_a\,\mathrm
  d\mu_f(z)\\&+\vO(r^{-1-2\mu})\quad =\quad\vO(r^{-1-2\mu}).   
  \end{split}\label{eq:bndInterclust}
\end{align}   Note that 
 the contribution from the integral is   of order
$-2$, cf. \eqref{eq:rCalc1}.  

In combination with \eqref{eq:smoothBnd1} we then
obtain  that
\begin{equation*}
  \int_0^\infty\,\norm {\phi(t)}^2\,\d t<\infty;\quad \phi(t):= \parb{f_1(H_a)-f_1(H)}S_a\psi(t).
\end{equation*} Hence   $\norm {\phi(t)}^2$  is integrable, and since
it has a bounded
time-derivative, we conclude that  $\norm{\phi(t)}\to 0$ for $t\to
\infty$. We are done.
\end{proof}

\subsection{Proof of Lemma \ref{lemma:asympt-compl}}\label{subsec:Proof of Lemma ref{lemma:asympt-compl}}
We claim that  the observables $S_a$   given by \eqref{eq:gDef} work
in Lemma \ref{lemma:asympt-compl}.  Thanks to  Lemma
\ref{lemma:phase-space-part1},  \eqref{eq:part} holds, and it  remains to show \eqref{eq:part2}.

 Using  any  $f_2\succ f_1 $,   we first record that 
\begin{equation}\label{eq:bndInterclust2}
  T_1:=\i \parb{\breve f_2(H_a)-\breve f_2(H)}S_af(H)=\vO(r^{-1-2\mu}),\quad \breve f_2(h):=hf_2(h),
\end{equation} 
        cf. \eqref{eq:bndInterclust}.
 Letting   $\phi_1(t)=\e^{-\i t
  H_a}f_1(H_a)\phi$, $\phi\in \vH$,  $\psi_1(t)=\e^{-\i t
  H}f_1(H)\psi$ and $\bD_2 S=\i[\brf_2(H),S]$, we   compute the time-derivative of
$\inp{\phi_1(t),S_a\psi(t)}$ as $\inp{\phi_1(t),T\psi_1(t)}$ with
\begin{align*}
   T&=\i\parb{\breve
                                    f_2(H_a)S_a-S_a\breve
                                    f_2(H)}f(H)=T_1+T_2+T_3;\\&T_2:=f_2(H_a)\parb{\bD_2 g_+(M_{a_0})}M_af(H),\quad T_3:=f_2(H_a)g_+(M_{a_0})\parb{\bD_2 M_{a_0}}f(H) .
\end{align*}

We substitute  this decomposition in the inner product, split it into three terms and
integrate each of them.  For  the contribution from  $T_1$,     in
turn we substitute \eqref{eq:bndInterclust2} allowing us then to   integrate
thanks to  \eqref{eq:smoothBnd1}. Note that we can use \eqref{eq:smoothBnd1}
for the  state  $\psi_1(t)$
as well as for   $\phi_1(t)$ 
  (recalling here and below that \eqref{eq:smoothBnd1} and
  \eqref{eq:smoothBnd2} are  also available for  $H_a$), yielding in
  particular the  bound
  $C\norm{\phi}$. This is  standard application of 
  Kato's  smooth operator theory, \cite{Ka}, and the Cauchy criterion
  is verified  by  bounding the time-integral at infinity
  only  (yielding in this case $C$ as small as wanted). 

For $T_2$ and $T_3$ we  compute using
  \cite[Lemmas 6.9 and  6.11]{Sk1} (applied to   $M_{a_0}$
  rather than to  $ B$) and \eqref{eq:rCalc2}
  \begin{align*}
    T_2&=4\,
    g_+'(M_{a_0})f_2(H_a)\sum_{b\in
         \vA'}\,p\cdot\parb{\chi^2_+(\abs{x})\nabla^2m_b(x)} p
         f(H)M_a+\vO(r^{-2\mu-1}),\\
T_3&=4\,
    g_+(M_{a_0})f_2(H_a)\,p\cdot\parb{\chi^2_+(\abs{x})\nabla^2m_a(x)} p  f(H)+\vO(r^{-2\mu-1}).
  \end{align*}   

The treatment  of $T_3$ is very similar to arguments
  used  in Step II of
the proof of Lemma \ref{lemma:phase-space-part1}. In fact we use 
again $Q= Q(a,j)$
(given as in \eqref{eq:smoothBnd2} and  \eqref{eq:Qdef}). We substitute
\eqref{eq:Hes0}  and  pull the appearing (components of) factors of
$Q^*$ to the left in the inner product. Thanks to
\eqref{eq:smoothBnd1} and the bound $[Qf_2(H_a),g_+(M_{a_0})]=\vO(r^{-3/2})$ 
 it then suffices to consider the contribution to the time-integral with
 explicitly appearing  factors of
 $Qf_2(H_a)\phi_1(t)$ and $Qf(H)\psi_1(t)$, yielding with
 \eqref{eq:smoothBnd2}  the  bound  
$C\norm{\phi}$ (with $C$ small   for the contribution at
infinity).

The treatment  of $T_2$ is similar, although we need
 an additional  commutation to the right.  More precisely (using Subsection \ref{subsec: Geometric
   considerations for a0=a{max}} with
 $a$ replaced by $b$) we note     that thanks to \eqref{eq:rCalc3} applied to
 $Q=Q(b,j)$ 
\begin{equation*}
  [Q,f(H)M_a]=[Q,f(H)]M_a +f(H)[Q,M_a]=\vO(r^{-3/2}). 
 \end{equation*} 
 Since $f(H)M_a$ is bounded  we can then  proceed as for  $T_3$ using 
 \eqref{eq:smoothBnd2} for $Q\psi_1(t)$
  and $Q\phi_1(t)=Qf_2(H_a)\phi_1(t)$.
\qed

\appendix
\section{Proof of \eqref{eq:MicroB0}}\label{sec:Strong bounds}

We shall in this appendix prove \eqref{eq:MicroB0} using a
time-dependent version of \cite[Lemma 5.3]{AIIS}. Recall that
$\lambda\in \R\setminus\vT_{\p}(H)$ is fixed. Our proof  based on the
Mourre estimate \eqref{eq:mourreFull} with $c=\sigma^2$, recalling \eqref{eq:optimal} and that we 
 have taken
$\sigma=\tfrac{10}{11}\sigma_0$ with $\sigma_0=\sqrt{d(\lambda,H)}$.
We will use  it on the form of \cite[Corollary
6.8]{Sk1} (or more precisely  \cite[Corollary 2.11]{AIIS}), i.e. 
\begin{equation}\label{eq:mourreFull9}
  \begin{split}
  \forall \text{ real } f\in C_\c^\infty(U_\lambda):\quad{f}(H)
  &\i[H,B]{f}(H)\\&\geq
 \,{f}(H)r^{-1/2}\parb{4\sigma^2-B^2}r^{-1/2}{f}(H)- Cr^{-2}.
  \end{split}
\end{equation} 
Below we will also use the parameter $\mu\in (0,1/2)$ from Condition
\ref{cond:smooth2wea3n12}.
\begin{lemma} \label{lemma:propa}  
Let  $\lambda\in \R\setminus\vT_{\p}(H)$, $\sigma'\in (0,\sigma)$  and  $\mu'\in(0,\mu)$.
  Consider 
   any real $f\in
 C_\c^\infty(U_\lambda)$ (with $U_\lambda$ given as in \eqref{eq:mourreFull9}) and  any   
$f_1\succ f$, also supported in $U_\lambda$,  
    \begin{align*}
      \Psi_\epsilon&=- f(H)
      r^{\mu'}F_\epsilon( B)r^{\mu'}f(H);\\&\quad
F_\epsilon(b)=(2\sigma'+3\epsilon-b)^{2\mu'}\chi^2_\epsilon(b-2\sigma'),\quad
                                             \chi_\epsilon(s)=\chi_-(s/\epsilon),\quad \epsilon>0
. 
\end{align*}
 There  exist $\kappa>0$ 
such that for all small $\epsilon>0$, as a quadratic form  on $L^2_{\mu'}$,
\begin{align}\label{eq:171118v}
  \begin{split}
   \i[H,\Psi_\epsilon]
\ge 
 \kappa &f(H)\chi_\epsilon(B-2\sigma')r^{-1+2\mu'}\chi_\epsilon(B-2\sigma')f(H)\\
&\quad\quad\quad\quad\quad\quad\quad\quad-C_\epsilon f_1(H)r^{-1-2(\mu-\mu')}f_1(H);\quad
C_\epsilon>0.
 \end{split}
\end{align} 
\end{lemma}
\begin{proof}  
Letting  $ \brf_1(H):=Hf_1(H)$, $\bD_1 T=\i[\brf_1(H),T]$
 and  
$\theta_\epsilon:=\sqrt{-F_\epsilon'}$ we   compute  using
 \cite[Lemmas  6.3 and 6.12]{Sk1}, cf. \eqref{eq:deM},
\begin{align*}
  f(H)r^{\mu'}\parb{ \bD_1 F_\epsilon(B)}&r^{\mu'}f(H)\\
= -r^{\mu'}\theta_\epsilon(B)&f(H)\i[H,B]f(H)\theta_\epsilon(B)r^{\mu'}+
   \vO(r^{2\mu'-2\mu-1}).
\end{align*}  It is an
elementary fact   that 
\begin{align*}
  \begin{split}
   -\tfrac12(2\sigma'+3\epsilon-b)^{1-2\mu'}F_\epsilon'(b)&=\mu'\chi^2_\epsilon(b-2\sigma')-(2\sigma'+3\epsilon-b)(\chi_\epsilon\chi'_\epsilon)(b-2\sigma')\\&\geq
                                                                                                                                                          \mu'\chi^2_\epsilon(b-2\sigma')\geq0. 
  \end{split}
\end{align*}
 We also compute
   \begin{align*}
   \bD_1
   {r}^{\mu'}=\mu '(\brf_1)'(H)r^{\mu'-1}B+\vO(r^{\mu'-2}).
 \end{align*}

Introducing 
\begin{equation*}
  S=\chi_\epsilon(B-2\sigma')r^{\mu'-1/2}f(H)\mand T=(2\sigma'+3\epsilon-B)^{\mu'-1/2}S,
\end{equation*}
this leads with \eqref{eq:mourreFull9} (and commutation) to the lower bound
\begin{align}\label{eq:Tbnd}
  \begin{split} \i[H,\Psi_\epsilon]  &\geq  2\mu'T^*\parb{4\sigma^2-B^2-B(2\sigma'+3\epsilon-B)}T+\vO(r^{2\mu'-2\mu-1})\\\
&\geq  2\mu'T^*\parb{4\sigma^2-(2\sigma'+2\epsilon)(2\sigma'+3\epsilon)}T+\vO(r^{2\mu'-2\mu-1})\\
&\geq \kappa_1T^*T+\vO(r^{2\mu'-2\mu-1});\quad\quad\quad\quad
\kappa_1=4\mu'\parb{\sigma^2-\sigma'^2}
>0. 
\end{split}
\end{align}

We claim that for some $\kappa_2,C_1>0$
\begin{align}\label{eq:lowST}
   T^*T\geq \kappa_2S^*S-C_1 r^{2\mu'-3}.
\end{align}  The combination of \eqref{eq:Tbnd} and
\eqref{eq:lowST} completes (after a commutation) the proof of
\eqref{eq:171118v} (with $\kappa=\kappa_1\kappa_2$).

Showing  \eqref{eq:lowST} amounts to  removing  the  factor
$(2\sigma'+3\epsilon-B)^{\mu'-1/2}$ of  $T$. 
 We write $S=Sf_1(H)$ and note that $ [S,f_1(H)]=\vO(r^{\mu'-3/2}).$ 
 Using the notation $\|\cdot\|_s=\|\cdot\|_{L^2_s}$ this leads to
 \begin{align*}
   \norm{S\psi}\,\leq \,\norm{ f_1(H)S\psi}+C_2\norm{\psi}_{\mu'-3/2}
\,\leq\, (2\kappa_2)^{-1}\norm{T\psi}+C_2\norm{\psi}_{\mu'-3/2},
 \end{align*} and therefore indeed to \eqref{eq:lowST}.
\end{proof} 
\begin{corollary}\label{cor:proof-eqref} Under the conditions of
  Lemma \ref{lemma:propa}  let $\psi(t)=\e^{-\i t
  H}f(H)\psi$, $\psi\in L^2_{\mu'}$. Then
\begin{subequations}
  \begin{equation}\label{eq:intB1}
  \int_0^\infty\,\norm[\big]{r^{\mu'-1/2}\chi_\epsilon(B-2\sigma')\psi(t)}^2\,\d t\,< \infty,
 \end{equation}  and 
\begin{equation}\label{eq:intB2}
   \sup_{t>0}
   \norm{r^{\mu'}\chi_\epsilon(B-2\sigma')}_{\psi(t)}\,< \infty.
 \end{equation}
\end{subequations}  
\end{corollary} 
\begin{proof} Let $\psi_1(t)=\e^{-\i t
  H}f_1(H)\psi$. Since $\inp{-\Psi_\epsilon}_{\psi(0)}< \infty$ it follows from
  \eqref {eq:smoothBnd1} and 
  \eqref{eq:171118v} (by differentiation and integration) that for any $\tau>0$
  \begin{align*}
    &\inp{-\Psi_\epsilon}_{\psi(\tau)}+\kappa \int_0^\tau\,\norm[\big]{r^{\mu'-1/2}\chi_\epsilon(B-2\sigma')\psi(t)}^2\,\d
    t\\& \quad\quad\quad\quad\leq
    \inp{-\Psi_\epsilon}_{\psi(0)}+C_\epsilon\int_0^\infty\,\norm[\big]{r^{-1/2-(\mu-\mu')}\psi_1(t)}^2\,\d
    t<\infty.
  \end{align*} Both terms to the left  are non-negative, yielding 
  in particular  \eqref{eq:intB1}. The bound \eqref{eq:intB2}  follows from
  the same estimate by a  commutation.
\end{proof}
\begin{corollary}\label{cor:proof-eqref2} The asymptotics
  \eqref{eq:MicroB0} holds.
\end{corollary}
\begin{proof} By a density argument, it  suffices   to consider    $\psi(t)=\e^{-\i t
  H}f(H)\psi$ with   $f\in
 C_\c^\infty(U_\lambda)$ real and with   $\psi\in L^2_{\mu'}
 \,(\subseteq \vH)$, henceforth fixed.  
We will then use
\eqref{eq:intB1} and \eqref{eq:intB2}
  with   $\sigma'$ taken close to
$\sigma$, guaranteeing that
\begin{equation}\label{eq:supPRop}
  \chi_\epsilon(b-2\sigma')=1\text{ on } \supp(\chi_-(\cdot/\sigma_0)).
\end{equation}  

We need to show that
\begin{equation*}
          L:=\lim_{t\to \infty}\,\norm{\chi_-(B/\sigma_0)\psi(t)}^2=0.
        \end{equation*}
First, by differentiating
$\norm{\chi_-(B/\sigma_0)\psi(t)}^2$ with respect to
$t$ and integrating,  it follows by using 
\cite[Lemmas 6.9 and 6.11]{Sk1},  \eqref{eq:intB1} and 
\eqref{eq:supPRop}, that indeed the limit $L$ exists.

 It remains to show that  $L=0$. From \eqref{eq:intB2} it follows that
  \begin{equation}\label{eq:rsmall}
    \norm[\big]{ 1_{\set{r\geq t}} \chi_\epsilon(B-2\sigma')\psi(t) }=\norm[\big]{ \parb{1_{\set{r\geq t}}r^{-\mu'}} r^{\mu'}\chi_\epsilon(B-2\sigma')\psi(t)) }\leq  C t^{-\mu'}.
  \end{equation}
From \eqref{eq:intB1} it follows that
\begin{equation*}
   \int_0^\infty\,\norm[\big]{1_{\set{r< t}} r^{-1/2}\chi_\epsilon(B-2\sigma')\psi(t)}^2\,\d t< \infty,
 \end{equation*} and therefore in particular, that  along some sequence
 $t_n\to \infty$
 \begin{equation*}
   t_n\norm[\big]{1_{\set{r< t_n}}
     r^{-1/2}\chi_\epsilon(B-2\sigma')\psi(t_n)}^2\to 0.
 \end{equation*} Then obviously also 
\begin{equation*}
   \norm[\big]{1_{\set{r< t_n}}
     \chi_\epsilon(B-2\sigma')\psi(t_n)}^2 \to 0,
 \end{equation*} which combined with \eqref{eq:supPRop} and \eqref{eq:rsmall} 
 yields that $\norm{\chi_-(B/\sigma_0)\psi(t_n)}\to 0.$
Hence   $L=0$  as wanted.
\end{proof}


\begin{thebibliography}{DoGa}


  \bibitem[AIIS]{AIIS}
T. Adachi, K. Itakura, K. Ito, E. Skibsted, 
\emph{New methods in spectral theory of $N$-body Schr{\"o}dinger
  operators},  Rev. Math. Phys. \textbf{33}  (2021), 48 pp. 


\bibitem[De]{De}
J. Derezi\'nski, \emph{Asymptotic completeness for $N$-particle long-range quantum
    systems}, Ann. of Math.  \textbf{38}   (1993),
  427--476.


\bibitem[DG]{DG}
J. Derezi{\'n}ski, C. G{\'e}rard, \emph{Scattering theory of
 classical and quantum {$N$}-particle systems}, Texts and Monographs in
  Physics,  Berlin, Springer 1997.


  
  








\bibitem[Gra]{Gra} G.M. Graf, \emph{Asymptotic completeness for
    $N$-body short-range quantum systems: a new proof},
  Commun. Math. Phys. \textbf{132} (1990), 73--101.

\bibitem[Gri]{Gri} M. Griesemer, \emph{$N$-body quantum systems with
    singular potentials}, Ann. Inst. Henri Poincar{\'e} {\bf 69} no.~2
  (1998), 135--187. 

  \bibitem[HS1]{HS} W. Hunziker, I. Sigal, \emph{Time-dependent
      scattering theory of $N$-body quantum systems},  Rev. Math. Phys. \textbf{12} (2000), 1033--1084.
  
\bibitem[HS2]{HS2} W. Hunziker, I. Sigal, \emph{
The quantum $N$-body problem}, 
J. Math. Phys. \textbf{41} (2000), 3448--3510.


    
 




\bibitem[Is]{Is}
 H. Isozaki, \emph{Many-body Schr{\"o}dinger equation--Scattering Theory and Eigenfunction
Expansions}, Mathematical Physics Studies, Singapore, Springer  2023. 







\bibitem[Ka]{Ka} T. Kato, \emph{Smooth operators and commutators},
  Studia Math. \textbf{31} (1968), 535--546.




\bibitem[RS]{RS}
M.~Reed, 
B.~Simon, \emph{Methods of modern mathematical physics {I}--{I\hspace{-.1em}V}},
  New York, Academic Press 1972-78.

\bibitem[SS]{SS}
I. Sigal, A. Soffer, \emph{The $N$-particle scattering problem:
  asymptotic completeness for short-range systems},  Ann. of Math. (2) \textbf{126} (1987), 35--108.



\bibitem[Sk1]{Sk1}
 E. Skibsted, \emph{Stationary scattering theory: the $N$-body
   long-range case},  Commun. Math. Phys. \textbf{401} (2023),  2193--2267.

\bibitem[Sk2]{Sk2}
E. Skibsted, \emph{Stationary completeness: the $N$-body
   short-range case}, preprint 8 April 2024,
 http://arxiv.org/abs/2306.07080v2, submitted to Adv.  Math. 

\bibitem[Ta]{Ta} H. Tamura, \emph{Asymptotic  completeness  for
    four-body  Schr\"odinger operators   with  short-range
    interactions},  Publ. Res. Inst. Math. Sci. \textbf{29} (1993), 1--21.
  



  

\bibitem[Ya1]{Ya1} D.R. Yafaev, \emph{Radiation conditions and
    scattering theory for $N$-particle Hamiltonians},
  Commun. Math. Phys. \textbf{154}  (1993), 523--554.

\bibitem[Ya2]{Ya2} D.R. Yafaev, \emph{Resolvent estimates and scattering
    matrix for $N$-body Hamiltonians},
  Intgr. Equat. Oper. Th. \textbf{21} (1995), 93--126.

\bibitem[Ya3]{Ya3} D.R. Yafaev, \emph{Eigenfunctions of the continuous
   spectrum for $N$-particle Schr{\"o}dinger
  operator}, Spectral and scattering theory, Lecture notes in pure and
applied mathematics, Marcel Dekker, New York  1994, 259--286.
 

\bibitem[Zi]{Zi} L.  Zielinski, \emph{A proof of asymptotic
    completeness for $n$-body  Schr{\"o}dinger
  operators}, Comm. Partial Differential Equations \textbf{19} (1994), 455--522.

\end{thebibliography}
\end{document}